\documentclass[a4paper, 12pt]{article}
\usepackage{CJK}
\usepackage{indentfirst}
\usepackage{amsmath, amsfonts, amsbsy, amsthm, amscd, graphicx}
\usepackage{amssymb, latexsym, color}
\usepackage{cases}
\numberwithin{equation}{section}
\theoremstyle{plain}
 \newtheorem{Theorem}{Theorem}[section]
 \newtheorem{Proposition}[Theorem]{Proposition}
 \newtheorem{Lemma}[Theorem]{Lemma}
 \newtheorem{Corollary}[Theorem]{Corollary}
 \theoremstyle{definition}
 \newtheorem{Definition}[Theorem]{Definition}
  \theoremstyle{remark}
 \newtheorem{Remark}{Remark}
 \newtheorem{Example}{Example}
 \newtheorem{Claim}{Claim}
 
\title{ Computation of the index of some meromorphic functions of degree $3$ on tori}

\author{ SARENHU}

\begin{document}
\maketitle

\noindent
\begin{abstract}
The index of a meromorphic function $g$  on a compact Riemann surface is an invariant of $g$, which is defined as the number of negative eigenvalues of the differential operator  $L:=-\Delta-|dG|^2$, where $\Delta$ is the Laplacian with respect to a conformal metric $ds^2$ 
on the Riemann surface, 
$G\colon M\to S^2$ is the holomorphic map corresponding to $g$. 
We consider the meromorphic function $w$ on the Riemann surface 
$M_a=\left\{(z,w)\in\widehat{\mathbb{C}}^2 \mid w^2=z(z-a)\left(z+\frac{1}{a}\right)\right\}(a\geqslant 1 )$
 homeomorphic to a torus, and we determine the index of $tw$ for all $a$ in the range $1\leqslant a\leqslant a_0$ (where $a_0$ can be numerically evaluated) and all $t>0$. 
\end{abstract}

\noindent
\section{\bf Introduction}

The index of a nonconstant meromorphic function $g$ on a compact Riemann surface is an invariant of $g$, which is defined as the number of negative eigenvalues of the differential operator $L:=-\Delta-|dG|^2$, where $\Delta$ is the Laplacian with respect to a conformal metric $ds^2 =\lambda d\zeta d\bar{\zeta}$ on the Riemann surface, defined by $\Delta: =\frac{4}{\lambda}\frac{\partial^2}{\partial \zeta\partial\bar{\zeta}}$ using a local coordinate $\zeta$, $G\colon M\to S^2$ is the holomorphic map corresponding to the meromorphic function $g$ and $|dG|$ is the norm of the differential $dG$ of $G$. The multiplicity of the eigenvalue $0$ of $L$ is called the nullity of $g$ and denoted by $\mathrm{Nul}(g)$. The operator $L$ depends on how to choose a conformal metric, but the index and the nullity do not depend on how to choose a conformal metric. 

The index of a meromorphic function is closely related to the index (Morse index) of a complete minimal surface with finite total curvature. Huber$\cite{Huber}$ and Osserman$\cite{Osserman}$ proved if the total curvature of a complete oriented minimal surface in $\mathbb{R}^3$ is finite, this minimal surface is identified with a Riemann surface given by excluding finite points from a compact Riemann surface, and the Gauss map on this minimal surface is extended to a meromorphic function on the compact Riemann surface. Fischer-Colbrie$\cite{ Fischer-Colbrie}$ and Gulliver-Lawson $\cite{Gulliver},\cite{Gulliver-Lawson}$ proved that for a complete oriented minimal surface  in $\mathbb{R}^3$, 
the index  is finite if and only if the total curvature is finite. This is a qualitative study of the  index. Fischer-Colbrie proved when the total curvature a complete oriented minimal surface in $\mathbb{R}^3$ is finite, the index coincides with the index of the extended Gauss map of this minimal surface. Tysk $\cite{Tysk}$ proved the index of a complete oriented minimal surface in $\mathbb{R}^3$ is bounded from above by some scalar multiple of the total curvature. This is the first quantitative study of the relationship between the index and the total curvature. Study of lower bound of index was done by Choe $\cite{Choe}$ and Nayatani $\cite{Nayatani-}$. Nayatani $\cite{Nayatani}$ studied for the index and the nullity of the operator $L_g$ associated to any meromorphic function $g$ on a compact Riemann surface $M$, how they change under a certain deformation $g_t$ of $g$ ($t$ is a positive real number). He considered the derivative $\wp^\prime$ of the Weierstrass $\wp$-function corresponding to the square lattice $\mathbb{Z}\oplus i\mathbb{Z}$ as a meromorphic function $g$,  and computed the index of $g_t$ when $t$ is sufficiently small and the nullity of $g_t$ for all $t$. In particular, he showed that there are two values $t_1,t_2(t_1<t_2)$ of $t$ such that the nullity is $4$. Furthermore, he investigated the change of index when $t$ becomes large. He showed that the indices of $t_1g, t_2g$ are $5$, and since $t_2g$ is the Gauss map of the Costa surface, he could conclude that the index of the Costa surface is $5$. 

In this paper, we study the index $\mathrm{Ind}(g)$ and the nullity $\mathrm{Nul}(g)$ of 
certain nonconstant meromorphic functions $g$ from a compact Riemann surface $M$ to $\widehat{\mathbb{C}}=\mathbb{C}\cup\{\infty\}$. 
In order to compute $\mathrm{Nul}(g)$, we recall the real vector space $H(g)$(see \eqref{H(g)}) which was introduced by Ejiri-Kotani $\cite{Ejiri-Kotani}$ and Montiel-Ros $\cite{Montiel-Ros}$. 
By the formula $\mathrm{Nul}(g)=3+\dim_\mathbb{R}H(g)$, we can compute $\mathrm{Nul}(g)$. If the genus of the Riemann surface $M$ is $1$, that is, $M$ is homeomorphic to the torus, $H(g)$ can be defined as follows. 
\begin{align*}
H(g)=&\Big\{f\omega\mid f:M\to \widehat{\mathbb{C}}\,\,\mbox{is a meromorphic function}, \\ 
 &\hspace{5mm}D(f)+\widetilde{B}(g)\geqslant 0, \mathrm{Res}_{p_i}\left(f\omega\right)=0, i=1,\dotsb ,\mu,\notag\\  
 &\hspace{5mm}\mathrm{Re}\int_{\alpha}(1-g^2, i(1+g^2), 2g)(f\omega)=0, \forall\alpha \,\,\mbox{closed curve}\Big\},  
\end{align*}
where $\omega$ is a fixed nonzero holomorphic one-form on $M$, $p_1, \dotsb, p_\mu$ are the ramification points of $g$ with ramification indices $e_1,\dotsb,e_\mu$, and $D(f)$ is the divisor of $f$, and when $P(g)$ is the polar divisor of $g$, $\widetilde{B}(g)$ is the divisor defined by $\widetilde{B}(g)=\sum_{i=1}^{\mu}e_ip_i-2P(g)$. 
Since the definition of $H(g)$  is complicated as it includes the period condition, we introduce the complex vector space 
\begin{align*}
\widehat{H}(g)=&\Big\{f\omega \mid f:M\to \widehat{\mathbb{C}}\,\,\mbox{is a meromorphic function},\notag\\ 
 &\hspace{20mm}D(f)+\widetilde{B}(g)\geqslant 0, \mathrm{Res}_{p_i}\left(f\omega\right)=0, i=1,\dotsb , \mu  \Big\}  
\end{align*}
that is easier to handle, excluding the period condition. $H(g)$ is a real subspace of $\widehat{H}(g)$. For $t\in\mathbb{C}\setminus\{0\}$, $H(tg)\neq H(g)$ in general, but $\widehat{H}(tg)=\widehat{H}(g)$.

As already mentioned, Nayatani$\cite{Nayatani}$ computed the index of Costa surface.  The compact Riemann surface of the Costa surface is $\mathbb{C}$ divided by  the square lattice $\mathbb{Z}\oplus i\mathbb{Z}$, which is homeomorphic to a torus, and the Gauss map of Costa surface is a scalar
multiple of the derivative $\wp^\prime$ of the Weierstrass $\wp$-function. $\mathbb{C}/{\mathbb{Z}\oplus i\mathbb{Z}}$ is isomorphic to 
\begin{align*}
M_1=\left\{(z,w)\in\widehat{\mathbb{C}}^2 \mid w^2=z(z^2-1)\right\}
\end{align*}
as a Riemann surface, and $\wp^\prime$ coincides with the meromorphic function $w$ except for a scalar multiple. As a generalization of Nayatani's setting, we consider a one-paramenter family 
\begin{align*}
M_a=\left\{(z,w)\in\widehat{\mathbb{C}}^2 \mid w^2=z(z-a)\left(z+\frac{1}{a}\right)\right\}, \,a\geqslant 1 
\end{align*}
of Riemann surfaces homeomorphic to a torus and the meromorphic function $w$, and we tackled the problem of computing the index and nullity of $tw$, $t>0$. We note that $M_a$ is isomorphic, as a Riemann surface, to $\mathbb{C}$ divided by the rectangular lattice $\mathbb{Z}\oplus ic\mathbb{Z},\,c>0$. As a result, we are able to determine $\mathrm{Ind}(tw)$ for all $a$ in the range $1\leqslant a\leqslant a_0$ (where $a_0$ can be numerically evaluated) and all $t>0$. This is the main theorem of this paper. 

 First we determine $\widehat{H}(w)$. Then we determine $t>0$ so that the dimension of $H(tw)$
is $1$ or more, and find exactly two values $t=t_1(a), t_2(a)\,\, (t_1(a) < t_2(a))$ for each $a$ in the range  $1\leqslant a\leqslant a_0$ (where $a_0$ can be numerically evaluated). Therefore, we can determine $\mathrm{Nul}(tw)$. 
By using the fact that  $\mathrm{Nul}(tw)$ also changes if $\mathrm{Ind}(tw)$ changes when $a, t$ move, we determine $\mathrm{Ind}(tw)$ for all $a$ in the range $1\leqslant a\leqslant a_0$ (where $a_0$ can be numerically evaluated) and all $t>0$.  
\noindent
\begin{Theorem}\rm (Theorem \ref{theorem-d})
If $t_1(a)$ and $t_2(a)$ are as described above,  then
\begin{align*}
\mathrm{Ind}(tw)=
\begin{cases}
5, & 0 < t\leqslant t_1(a),t_2(a)\leqslant t, \\
6, &  t_1(a)< t< t_2(a) 
\end{cases}
\end{align*}
 for any $a$ in the range $1\leqslant a \leqslant a_0$ ( where $a_0$ can be numerically evaluated ). 
\end{Theorem}

This paper is organized as follows. In Section 2, we define the index and nullity of a meromorphic function on a compact Riemann surface. We also recall the vector spaces which are used in the computation of nullity and were introduced by Ejiri-Kotani \cite{Ejiri-Kotani} and Montiel-Ros \cite{Montiel-Ros}. In Section 3 we consider a certain family of meromorphic functions $g_a$ of degree three defined on Riemann surfaces $M_a$, $a\geqslant1$, homeomorphic to the torus, and describe the above vector spaces in these special cases. 
In Section 4 we compute the nullity of $tg_a$ for all $t>0$ and $a$ in the range $1\leqslant a\leqslant a_0$, where $a_0$ is a constant which can be numerically evaluated.  
In Section 5 we compute the index of $tg_a$ for all $t$ and $a$ in the same range. 

\section*{Acknowledgments}
First, I want to thank Prof. Shin Nayatani for his academic guidance and personal supports, helping me to tackle this research. Secondly, I want to thank Prof. Ryoichi Kobayashi for many comments and advices. Thirdly, I want to thank  Dr. Takumi Gomyou for his help in completing this paper. Finally, I would like to thank my family for encouragement. Without their support this paper would not have been possible.

\noindent
\section{\bf Preliminaries}

In this section, we define the index and nullity of a meromorphic function on a compact Riemann surface. We also recall the vector spaces which are used in the computation of nullity. 

Let $M$ be a Riemann surface, and $g$ be a nonconstant meromorphic function from a compact Riemann surface $M$ to $\widehat{\mathbb{C}}$. We fix a conformal metric $ds^2 = \lambda dz d\overline{z}$, where $\lambda$ is a positive function on $M$, and consider the differential operator $L=-\Delta-|dG|^2$. Here, $\Delta: =\frac{1}{\lambda}\left(\frac{\partial^2}{\partial x^2}+\frac{\partial^2}{\partial y^2}\right)=\frac{4}{\lambda}\frac{\partial^2}{\partial z\partial \overline{z}}$ is the Laplace-Beltrami operator of $ds^2$. 
\begin{align*}
G=\left(\frac{2\mathrm{Re}(g)}{|g|^2+1},\frac{2\mathrm{Im}(g)}{|g|^2+1}, \frac{|g|^2-1}{|g|^2+1}\right): M \to S^2
\end{align*}
is the holomorphic map corresponding to the meromorphic function $g$, where $S^2$ is the unit sphere of $\mathbb{R}^3$. $\left|dG\right|^2$ is the square of the norm of $dG$ with respect to the metric $ds^2$, that is, $\left|dG\right|^2=\frac{1}{\lambda}\left(\left|\frac{\partial G}{\partial x}\right|^2+\left|\frac{\partial G}{\partial y}\right|^2\right)=\frac{4}{\lambda}\left|\frac{\partial G}{\partial z}\right|^2. $   

\noindent
\begin{Claim}\label{nul}
$N(g)=\left\{u\in C^\infty (M)\mid Lu=0\right\}$ does not depend on how the conformal metric $ds^2$ is taken. 
\end{Claim}
Actually, 
$\widetilde{L}=\phi L,\,\,
\phi=\frac{\lambda}{\widetilde{\lambda}}>0$ for two conformal metrics $ds^2=\lambda dzd\overline{z}, \,\,\widetilde{ds^2}=\widetilde{\lambda} dzd\overline{z}$. 
(Note that $\phi$ is globally defined on $M$.)
Therefore, $\widetilde{L}u=0$ if and only if $Lu=0$. 
\hfill$\qedsymbol$

We now define the index and the nullity of $g$. 
\noindent
\begin{Definition}
We define the index $\mathrm{Ind}(g)$ of the meromorphic function $g$ as 
the number of negative eigenvalues (counted with multiplicities) of the operator $L$, 
and define the nullity $\mathrm{Nul}(g)$ of $g$ as
\begin{align*}
\mathrm{Nul}(g)=\dim N(g). 
\end{align*} 
\end{Definition}

The nullity $\mathrm{Nul}(g)$ does not depend on the choice of conformal metric 
$ds^2$ by Claim \ref{nul}. 
The index $\mathrm{Ind}(g)$ also does not depend on the choice of $ds^2$ 
by the following discussion.

\newcommand{\C}{\mathbb{C}}
\newcommand{\R}{\mathbb{R}}

The bilinear form associated with $L$ is represented by $Q\colon C^\infty(M)\times C^\infty(M)\to\mathbb{R}$. That is, for $u,v\in C^\infty(M)$, we define $Q(u,v)$ as 
\begin{align*}
Q(u, v)&=\int_{M}(Lu)vdA\notag\\
&=\int_{M}(-\Delta u-\left\vert dG\right\vert^2u)vdA. 
\end{align*}
$Q$ is symmetric. 
\begin{align*}
Q(u,u)=\int_{M}\left(|du|^2-|dG|^2u^2\right)dA
\end{align*}  
for a function $u$ on $M$. 

\noindent
\begin{Remark}\label{invariance}
$Q$ does not depend on the choice of conformal metric $ds^2$ on $M$. 

Actually, 
if we let $ds^2=\lambda(dx^2+dy^2)$, we obtain 
\begin{align}\label{Qproof}
Q(u,u)
=\int_{M}\left(\left(\frac{\partial u}{\partial x}\right)^2+\left(\frac{\partial u}{\partial y}\right)^2-\left(\left|\frac{\partial G}{\partial x}\right|^2+\left|\frac{\partial G}{\partial y}\right|^2\right)u^2\right)dxdy. 
\end{align}
Since $\lambda$ is not included in the rightmost side of \eqref{Qproof}, $Q$ does not depend on the choice of conformal metric $ds^2$ on $M$. 
\end{Remark}

Let the eigenspace $V_\lambda$ corresponding to the eigenvalue $\lambda$ of $L$ be 
$V_\lambda=\left\{u\in C^\infty(M)\mid Lu=\lambda u\right\}$, 
$\mu_1<\mu_2<\dotsb<\mu_k<0$ be the set of all negative eigenvalues of $L$, and $V=V_{\mu_1}\oplus V_{\mu_2}\oplus\dotsb\oplus V_{\mu_k}\subset C^\infty (M)$. 
\noindent
\begin{Claim}\label{claim_negative}
$Q$ is negative definite on $V$, that is, $Q(u, u)<0$ for any $u\in V\smallsetminus\{0\}$. 
\end{Claim}
Actually, for any $u\in V$, $u$ can be written as 
$u=\sum_{i=1}^{k}u_i,\quad u_i\in V_{\mu_i}. $
Then 
\begin{align*}
Q(u,v)=\int_{M}(Lu)udA
=\sum_{i=1}^{k}\mu_i\int_{M}u_i^2dA, 
\end{align*}
Therefore, when $u\neq 0$, $Q(u, u)<0$. 
\hfill$\qedsymbol$

In fact, one can show that $V$ is a maximal subspace of $C^\infty (M)$ on which 
$Q$ is negative definite. 
Thus, the index $\mathrm{Ind}(g)$ coincides with the dimension of a maximal subspace 
of $C^\infty(M)$ on which $Q$ is negative definite. 
Hence, by Remark \ref{invariance}, $\mathrm{Ind}(g)$ does not depend on the choice of 
conformal metric $ds^2$.

Let $ds_{S^2}^{2}=\frac{4}{(1+|z|^2)^2}dzd\bar{z}$ be the standard Riemannian metric on $\widehat{\mathbb{C}}.$ Let $\zeta$ be a local holomorphic coordinate on $M$. The pull-back $ds_{g}^{2}$ of $ ds_{S^2}^{2}$ by $g$ can be written as $ ds_{g}^{2}=g^* ds_{S^2}^{2}=\lambda d\zeta d{\bar \zeta}$, where $\lambda=\frac{4|g^\prime|^2}{(1+|g|^2)^2}.$  Let $\Delta_g=\frac{4}{\lambda}\frac{\partial^2}{\partial \zeta\partial\bar{\zeta}}$ be the Laplacian of $ds_{g}^{2}$. If $L$ corresponding to this $ds_{g}^{2}$ is represented by $L_g$, $L_g=-\Delta_g-2$. 

Since $\lambda=0$ at the ramification points of $g$, $ds^2_g=0$ at the ramification points. Although $ds^2$ is not strictly a conformal metric, one can show that 
$\mathrm{Ind}(g)$ can be computed as the number of negative eigenvalues (counted with multiplicities) of $L_g$ and $\mathrm{Nul}(g)$ can be computed as the multiplicity of the eigenvalue $0$ of $L_g$. 
In other words, we have the following. 

\noindent
\begin{Lemma}\rm\label{def-Ind}
$\mathrm{Ind}(g)$ can be computed as the number of eigenvalues (counted with multiplicities) of $-\Delta_g$ which are smaller than $2$. 
 $\mathrm{Nul}(g)$ can be computed as the multiplicity of the eigenvalue $2$ of $-\Delta_g$. 
 \end{Lemma}

\begin{Remark}\label{Indg-bar}
Let $M^\prime$ be a complete oriented minimal surface in $\mathbb{R}^3$ and $ds^2$ be the first fundamental form on $M^\prime$. The operator $L$ corresponding to $ds^2$ becomes the Jacobi operator 
$L=-\Delta + 2K$. 
Here, $\Delta$ is the Laplacian corresponding to $ds^2$ and $K$ is the Gaussian curvature of $ds^2$. 
Then the {\em (Morse) index} of $M'$ is defined as 
\begin{align*}
\mathrm{Ind}(M^\prime)=\sup\left\{\dim V\mid V\subset C_{0}^{\infty}(M^\prime), \,Q\,\,\mbox{is negative definete on} \,\,V\right\}. 
\end{align*}
Here, $Q(u,v)=\int_{M^\prime}(Lu)vdA,\,u,v\in C_{0}^{\infty}(M^\prime)$. Furthermore, Fischer-Colbrie \cite{ Fischer-Colbrie} proved that when the total curvature of $M'$ is finite, the index of $M'$ coincides with the index of the extended Gauss map of $M'$. 
\end{Remark}

\begin{Example}
The index of the catenoid is $1$. 
In fact, the catenoid is identified with $\mathbb{C}-\{0\}$ as a Riemann surface, and its extended Gauss map is the meromorphic function $g(z)=z$ on $\widehat{\mathbb{C}}$. Therefore, $ds^2_g=ds^2_{S^2}$, the standard metric of the unit sphere. Thus, $\mathrm{Ind}(g)$ coincides with the number of eigenvalues (counted with multiplicities) of $-\Delta_{S^2}$ which are smaller than 2. Since $0$ is the only such eigenvalue and has multiplicity $1$, we conclude that $\mathrm{Ind}(g)=1$. Thus, the index of the catenoid is $1$ by Remark \ref{Indg-bar}. 
\end{Example}

\noindent
\begin{Proposition}\label{def-Nul 3} 
$\mathrm{Nul}(g)\geqslant 3$. 
\end{Proposition}

\begin{proof} 
We define $L(g)$ as $L(g)=\left\{a\cdot G\mid a\in\mathbb{R}^3\right\}$.  
Since $G=(G_1,G_2,G_3)\colon M\to S^2\subset \mathbb{R}^3$ is a holomorphic map, 
$G$ satisfies the harmonic map equation $\Delta G+\left|dG\right|^2G=0$. Therefore, $a\cdot G\in N(g)$. Thus, 
$L(g)\subset N(g)$.  
The dimension of $L(g)$ is three. In fact, if this is not true, we have a linear relation 
$a_1 G_1 + a_2 G_2 +a_3 G_3=0$, and this means that the image of $G$ lies in a great circle 
of $S^2$. But this implies $G$ is a constant map (and $g$ is a constant function) 
as it is holomorphic. This contradicts the assumption that $g$ is nonconstant.
Therefore, $\dim L(g)=3$ and $\dim N(g)\geqslant 3$. 
This completes the proof of the proposition. 
\end{proof}

As mentioned in the above proof, $L(g)\subset N(g)$, and $\mathrm{Nul}(g)> 3$ if and only if 
$N(g)\setminus L(g)\neq \emptyset$. 
In order to compute $\mathrm{Nul}(g)$, we recall the work of Ejiri-Kotani \cite{Ejiri-Kotani} 
and Montiel-Ros \cite{Montiel-Ros}. 
They observed that an element of $N(g)\setminus L(g)$ appears in the following way. 
\begin{Definition}
Let $X\colon M'=M\setminus \{p_1,\cdots, p_\mu\}\to \mathbb{R}^3$ be a complete branched minimal immersion 
of finite total curvature, where $M$ is a compact Riemann surface. 
An end $p_i$ is said to be {\em planer} if there exists a unit vector $a_i\in S^2$ such that 
$\langle X, a_i \rangle$ is bounded in a neighborhood of $p_i$. 
\end{Definition}
\begin{Proposition}[Ejiri-Kotani \cite{Ejiri-Kotani} and Montiel-Ros \cite{Montiel-Ros}] \rm
Let $g\colon M\to \widehat{\mathbb{C}}$ be a nonconstant meromorphic function on a compact Riemann surface $M$, 
and $G\colon M\to S^2$ be the holomorphic map corresponding to $g$. 
Let $X\colon M'=M\setminus \{p_1,\cdots, p_\mu\}\to \mathbb{R}^3$ be a 
complete branched minimal immersion of finite total curvature whose extended Gauss map is $g$ 
and whose ends are all planer. 
Then $u=\langle X, G \rangle\colon M'\to \mathbb{R}$, the support function of $X$, extends to $M$ smoothly 
and gives an element of $N(g)\setminus L(g)$. 
\end{Proposition}

On the contrary, Ejiri-Kotani \cite{Ejiri-Kotani} and Montiel-Ros \cite{Montiel-Ros} proved 
that any element of $N(g)\setminus L(g)$ appeared as the support function of a complete branched 
minimal surface with planer ends. 
\begin{Theorem}[Ejiri-Kotani \cite{Ejiri-Kotani} and Montiel-Ros \cite{Montiel-Ros}]
\label{Thm-$H$} \rm
For any $u\in N(g)\smallsetminus L(g)$, there exists a complete branched minimal immersion 
$X\colon M'=M\smallsetminus \{p_1,\dotsb,p_\mu\}\to\mathbb{R}^3$ whose ends are all planer and whose extended Gauss map coincides with $g$ such that $u=\langle X,G\rangle$ on $M$, where 
$G\colon M\to S^2$ is 
the holomorphic map corresponding to $g$. In terms of the Weierstrass representation formula, this 
assertion is stated as follows. 
Let $H(g)$ be the real vector space defined by 
\begin{align}\label{H(g)}
H(g) =&\Big\{\omega\in H^0(K(M) + \widetilde{B}(g))\mid\mathrm{Res}_{p_i}\omega= 0, \forall i=1,\dotsb,\mu, \notag\\ 
&\hspace{20mm}\mathrm{Re}\int_{\alpha}((1 -g^2), i(1 + g^2), 2g)\omega =0, \,\,\forall\alpha\in H_1(M, \mathbb{Z})  \Big\},  
\end{align}
where $K(M)$ is the canonical divisor of $M$, $p_i$ are the ramification points of $g$ with ramification indices $e_i$, $P(g)$ is the polar divisor of $g$ and  $\widetilde{B}(g)$ is the divisor defined by $\widetilde{B}(g)=\sum_{i=1}^{\mu}e_ip_i-2P(g)$. 
Then for any $u\in N(g)\smallsetminus L(g)$, there exists $\omega\in H(g)\setminus \{0\}$ such that 
$u=\langle X_\omega,G\rangle$ on $M$, where $X_\omega\colon M\smallsetminus\{p_1,\dotsb, p_\mu\}\to 
\mathbb{R}^3$ is the branched minimal immersion defined by 
\begin{align}
X_\omega (p)=\mathrm{Re}\int_{p_0}^{p}((1 -g^2), i(1 + g^2), 2g)\omega. 
\end{align}
In particular, we have the linear isomorphism $N(G)/L(G)\cong H(g)$. 
\end{Theorem}
\begin{Corollary}[Ejiri-Kotani \cite{Ejiri-Kotani} and Montiel-Ros \cite{Montiel-Ros}]\rm
\label{Thm-$Nul$}
$\mathrm{Nul}(g)$ can be computed from the dimension of $H(g)$ by the following formula$\colon$ 
\begin{align}
\mathrm{Nul}(g)-3=\dim_\mathbb{R}H(g). 
\end{align}
\end{Corollary}

The complex vector space $\widehat{H}(g)$ defined as follows plays an auxiliary role in the computation of $H(g)$. 
  
 \begin{Definition}\label{def-$HH(g)$}
Define a complex vector space $\widehat{H}(g)$ as 
\begin{align*}
\widehat{H}(g) = \left\{\omega\in H^0(K(M) + \widetilde{B}(g))\mid\mathrm{Res}_{p_i}\omega= 0, i=1,\dotsb, \mu\right\}. 
\end{align*}
\end{Definition}

For $t\in\mathbb{C}\setminus\{0\}$, $H(tg)\neq H(g)$ in general, but $\widehat{H}(tg)=\widehat{H}(g)$.

\noindent
\section{\bf Setting}

In this section, we consider a certain family of meromorphic functions of degree three defined on Riemann surfaces homeomorphic to the torus, and describe the above vector spaces in these special cases.

\noindent
\subsection{\bf Torus}

We first define the Riemann surfaces.
Let  
\begin{align*}
M_a=\left\{(z,w)\in\widehat{\mathbb{C}}^2 \mid w^2=z(z-a)\left(z+\frac{1}{a}\right)\right\},\quad a\geqslant 1. 
\end{align*}
If $(r_1,\theta_1), (r_2,\theta_2),(r_3,\theta_3)$ are the polar coordinates centered at $0,a,-\frac{1}{a}$, then $z\in\widehat{\mathbb{C}}$ is represented in three ways as 
$$z=r_1e^{i\theta_1}, z=a+r_2e^{i\theta_2}, z=-\frac{1}{a}+r_3e^{i\theta_3},\quad r_1,r_2,r_3\geqslant0,\,0\leqslant\theta_1,\theta_2,\theta_3<{2\pi}. $$
Define the two branches $w_1, w_2$ of $w$ by
$$w_1=\sqrt{r_1r_2r_3}e^\frac{i({\theta}_1+{\theta}_2+{\theta}_3)}{2},\,\,
w_2=\sqrt{r_1r_2r_3}e^{\frac{i({\theta}_1+{\theta}_2+{\theta}_3)}{2}+\pi}=-w_1. $$
Prepare two copies of the Riemann spheres $\widehat{\mathbb{C}}$, let them be $\widehat{\mathbb{C}}_1$ and $\widehat{\mathbb{C}}_2$, respectively, and consider $w_1$ as a function on $\widehat{\mathbb{C}}_1$ and $w_2$ as a function on $\widehat{\mathbb{C}}_2$. Put in a slit in the half line connecting $z=a$ and $z=\infty$ on $\widehat{\mathbb{C}}_1$, and let the upper part (of the slit) be $l_1$ and the lower part be $\widetilde{l}_1$. Put in a slit in the half line connecting $z=a$ and $z=\infty$ on $\widehat{\mathbb{C}}_2$, and let the upper part be $l_2$ and the lower part be $\widetilde{l}_2$.  Put in a slit in the line segment connecting $z=0$ and $z=-\frac{1}{a}$ on $\widehat{\mathbb{C}}_1$, and let the upper part be $h_1$ and the lower part be $\widetilde{h}_1$. Put in a slit in the line segment connecting $z=0$ and $z=-\frac{1}{a}$ on $\widehat{\mathbb{C}}_2$, and let the upper part be $h_2$ and the lower part be $\widetilde{h}_2$. If $\widetilde{l}_1$ is attached to $l_2$, $h_1$is attached to $\widetilde{h}_2$, $\widetilde{h}_1$ is attached to $h_2$ and $\l_1$ is attached to $\widetilde{l}_2$ so that $w_1$ and $w_2$ are continuously connected, a Riemann surface $M^\prime_a$  homeomorphic to the torus can be obtained. Let $z_1\in\widehat{\mathbb{C}}_1 ,  z_2\in\widehat{\mathbb{C}}_2$, and define the map $\phi :  M^\prime_a\to M_a$ by $z_1\longmapsto (z_1,w_1(z_1)), z_2\longmapsto (z_2,w_2(z_2))$. 
Let $z_1\in l_1\subset \widehat{\mathbb{C}}_1$ and $z_2\in \widetilde{l}_2\subset\widehat{\mathbb{C}}_2$, and suppose $z_1=z_2$ in $M^\prime_a$. Then $\phi(z_1)=\phi(z_2)$ since $w_1(z_1)=w_2(z_2)$. The same holds at the other slits. Therefore, the map $\phi$ is well-defined. It is confirmed that $\phi$ is bijective. Identify $M_a$ with $M^\prime_a$ by this bijection, and consider $M_a$ as a Riemann surface. 

\noindent
\subsection{\bf Vector spaces $H(w)$ and $\widehat{H}(w)$}

In this subsection, we describe the vector spaces, reviewed in the section 2, when the mermorphic function is $w\colon M_a\to \widehat{\mathbb{C}}$.

$z: M_a\ni (z,w)\longmapsto z\in\widehat{\mathbb{C}}$ is a meromorphic function of degree $2$ on $M_a$, has $(\infty,\infty)$ as a  pole of  order $2$, and has $(0,0)$ as a zero of order $2$. $dz$ is  a meromorphic differential, has $(\infty,\infty)$ as a pole of order $3$,and has  $(0,0), (a,0), (-\frac{1}{a},0)$ as  zeros of order $1$, respectively.  

$w: M_a\ni (z,w)\longmapsto w\in\widehat{\mathbb{C}}$ is a meromorphic function of degree $3$ on $M_a$, has $(\infty,\infty)$ as a pole of order $3$, and has $(0,0), (a,0), (-\frac{1}{a},0)$ as  zeros of order $1$, respectively.  
$dw$ is a meromorphic differential, has $(\infty,\infty)$ as a pole of order $4$, and has $(A_1,\pm B_1),(A_2,\pm B_2)$ as  zeros of order $1$, respectively. 
Here, $A_1=\frac{a-\frac{1}{a}+\sqrt{a^2+\frac{1}{a^2}+1}}{3},A_2=\frac{a-\frac{1}{a}-\sqrt{a^2+\frac{1}{a^2}+1}}{3},B_1$ is the value $w_1(A_1)$ of $w_1$ at $z=A_1\in \widehat{\mathbb{C}}_1$, and $B_2$ is the value $w_1(A_2)$ of $w_1$ at $z=A_2\in \widehat{\mathbb{C}}_1$. 

$\frac{dz}{w}$ has neither zero nor pole everywhere. 
Using the nowehre vanishing holomorphic differential $\frac{dz}{w}$, $H(w)$ can also be written as follows. 
\begin{align*}
H(w)=&\Big\{f\frac{dz}{w}\mid f:M_a\to\widehat{\mathbb{C}}\,\,\mbox{is a meromorphic function}, \notag\\
 &\hspace{15mm} D(f)+\widetilde{B}(w)\geqslant 0,\,\,\mathrm{Res}_{p_i}\left(f\frac{dz}{w}\right)=0,\, i=1,\dotsb , \mu, \notag\\
 & \hspace{15mm}\mathrm{Re}\int_{\alpha}((1-w^2), i(1+w^2), 2w)\left(f\frac{dz}{w}\right)=0,\,\forall\alpha \,\,\mbox{closed curve}\Big\},  
\end{align*}
where $D(f)$ is the divisor of $f$.  
Similarly, $\widehat{H}(w)$ can also be written as follows. 
\begin{align*}
\widehat{H}(w)=&\Big\{f\frac{dz}{w} \mid f:M_a\to \widehat{\mathbb{C}}\,\,\mbox{is a meromorphic function},\notag\\ 
&\hspace{20mm}D(f)+\widetilde{B}(w)\geqslant 0,\,\, \mathrm{Res}_{p_i}\left(f\frac{dz}{w}\right)=0,\, i=1,\dotsb , \mu  \Big\}. 
\end{align*}

\noindent
\section{\bf Computation of $\mathrm{Nul}(tw)$}

In this section we compute the dimension of $H(tw)$ and, as a consequence, compute the nullity of $tw$.  
First, we find a basis of $\widehat{H}(w)$. 
\begin{Lemma}\rm\label{widehatH}
$\widehat{H}(w)$
is a three dimensional complex vector space spanned by 
\begin{align*}
&\eta_1=\frac{z-\frac{1}{2}(A_1+A_2)}{(z-A_1)^2(z-A_2)^2}\frac{dz}{w}, \\
&\eta_2=\frac{z^2-A_1A_2}{(z-A_1)^2(z-A_2)^2}\frac{dz}{w}, \\
&\eta_3=\frac{z-\frac{1}{2}(A_1+A_2)}{(z-A_1)^2(z-A_2)^2}dz.  
\end{align*} 
\end{Lemma}
\begin{proof}
 First, we prove that $\eta_1,\eta_2,\eta_3$ are elements of $\widehat{H}(w)$. 
 Note that 
 \begin{align*}
\widetilde{B}(w)=2(A_1,\pm B_1)+2(A_2,\pm B_2)-3(\infty,\infty)
\end{align*}
by 
$$B(w)=(A_1,\pm B_1)+(A_2,\pm B_2)+2(\infty,\infty),\,\, P(w)=3(\infty,\infty). $$ 
We compute the divisors of 
\begin{align*}
& f_1=\frac{z-\frac{1}{2}(A_1+A_2)}{(z-A_1)^2(z-A_2)^2},\\ 
& f_2=\frac{z^2-A_1A_2}{(z-A_1)^2(z-A_2)^2},\\ 
& f_3=\frac{z-\frac{1}{2}(A_1+A_2)}{(z-A_1)^2(z-A_2)^2}w. 
\end{align*}
They are given by
\begin{align*}
&D(f_1)=-2(A_1,\pm B_1)-2(A_2,\pm B_2)+6(\infty,\infty)+(\frac{1}{2}(A_1+A_2),\pm B_3), \\
&D(f_2)=-2(A_1,\pm B_1)-2(A_2,\pm B_2)+4(\infty,\infty)+(\sqrt{A_1A_2},\pm B_4)+(-\sqrt{A_1A_2},\pm B_4), \\
&D(f_3)=-2(A_1,\pm B_1)-2(A_2,\pm B_2)+3(\infty,\infty)+(0,0)+(\frac{1}{2}(A_1+A_2),\pm B_3)+(a,0)\\
&\hspace{15mm}+(\frac{1}{a},0),  
\end{align*}
where $B_3$ is the value of $w_1$  at $z=\frac{1}{2}(A_1+A_2)\in\widehat{\mathbb{C}}_1$ and $B_4$ is the value of  $w_1$ at $z=-\sqrt{A_1A_2}\in\widehat{\mathbb{C}}_1$. 
Therefore, 
$$D(f_i)+\widetilde{B}(w)\geqslant 0,\quad i=1, 2 , 3. $$ 

We compute the residues of $\eta_i,\, i=1, 2 , 3$. 
Let $\zeta=z-A_1$ near $(A_1,\pm B_1)$. $w$ can be written as 
$$w=\pm B_1+b_2\zeta^2 +\dotsb$$
near $\zeta=0$. If we compute by using them, we obtain  
\begin{align*}
&\eta_1\sim \frac{1}{\pm 2B_1(A_1-A_2)\zeta^2}d\zeta, \\
&\eta_2\sim \frac{A_1(A_1+A_2)}{\pm B_1(A_1-A_2)^2\zeta^2}d\zeta, \\
&\eta_3\sim\frac{1}{2 (A_1-A_2)\zeta^2}d\zeta. 
\end{align*}
Therefore, 
$$\mathrm{Res}(\eta_i, (A_1,\pm B_1))=0,\quad i=1,2,3. $$ 
Similarly, we obtain
$$\mathrm{Res}(\eta_i, (A_2,\pm B_2))=0,\quad i=1,2,3. $$ 
 
Next, we prove that $\eta_1,\eta_2, \eta_3$ are linearly independent. 
For $\gamma_1, \gamma_2, \gamma_3\in\mathbb{C}$ we write 
\begin{align}\label{=0}
\gamma_1\eta_1+\gamma_2\eta_2+\gamma_3\eta_3=0.
\end{align}
  If we substitute $z=\frac{1}{2}(A_1+A_2)$ into $\eqref{=0}$, we obtain $\gamma_2=0$. Therefore, $\eqref{=0}$ becomes
\begin{align}\label{=1,2}
\gamma_1\eta_1+\gamma_3\eta_3=0. 
\end{align}
$\eqref{=1,2}$ can be written as   
\begin{align*}
\gamma_1(c_0+c_1\zeta+\dotsb)+\gamma_3(d_1\zeta+\dotsb)=0 ,\quad c_0\neq 0,\,d_1\neq 0,\,\,\mbox{near}\,\, (0,0). 
\end{align*} 
$\gamma_1c_0=0$ when $\zeta\to 0$.  $\gamma_1=0$ by $c_0\neq 0$. Therefore, $\gamma_3=0$. 
Thus, $\eta_1,\eta_2, \eta_3$ are linearly independent. 

Next, we prove that the complex dimension of $\widehat{H}(w)$ is $3$. 
If, $\dim_\mathbb{C}\widehat{H}(w)>3$, then $\dim_\mathbb{R}\widehat{H}(w)\geqslant 8$. Since there are only $6$ simultaneous equations from the period condition, the dimension of $H(w)$ is $2$ or more. By Corollary \ref{Thm-$Nul$}, $\mathrm{Nul}(tw)$ is $5$ or more for all $t$. 
On the other hand, 
by a result of Nayatani \cite[p518, THEOREM 2]{Nayatani}, 
to be reviewed below as Theorem \ref{N-d},   
the nullity of $tw$ is $3$ when $t$ is sufficiently small, which is a contradiction. Therefore, $\dim_\mathbb{C}\widehat{H}(w)=3$. 
\end{proof}

\begin{Theorem}[Nayatani \cite{Nayatani}]\label{N-d}\rm
Let $g\colon M\to \widehat{\mathbb{C}}$ be a nonconstant meromorphic function of degree $d$.  Let $\nu$ be the number of distinct poles of $g$. Then the following estimates hold for all sufficiently small $t\colon  $
$$\mathrm{Ind}(tg)\geqslant 2d-\nu, \quad
\mathrm{Ind}(tg)+\mathrm{Nul}(tg)\leqslant 2d+\nu+1, \quad
\mathrm{Nul}(tg) \leqslant2\nu+1.   
$$
In particular, if $\nu = 1$, then we have
$$\mathrm{Ind}(tg)=2d-1\quad\mbox{and}\quad\mathrm{NuI}(tg)=3$$
for all sufficiently small $t$.
\end{Theorem}

In the proof of the next lemma, we use the perfect elliptic integrals. 
So we recall the definition of these integrals.

\begin{Definition}
\begin{align*}
K(k)=\int_{0}^{\frac{\pi}{2}}\frac{d\theta}{\sqrt{1-k^2\sin^2\theta}} 
\end{align*}
is called the {\em perfect elliptic integral of the first kind} and
\begin{align*}
E(k)=\int_{0}^{\frac{\pi}{2}}\sqrt{1-k^2\sin^2\theta}d\theta
\end{align*}
is called the {\em perfect elliptic integral of the second kind}. 
\end{Definition}

\begin{Lemma}\rm\label{dimH}
For each $a$ in the range  $1\leqslant a\leqslant a_0$ (where $a_0$ can be numerically evaluated) there are  positive real numbers $t_1(a), t_2(a)$ ($t_1(a) < t_2(a)$)  such that 
\begin{align*}
\dim_ \mathbb{R}H(tw)=
\begin{cases}
1, & t=t_1(a), t_2(a),\\
0, &  t>0, t\neq t_1(a), t_2(a). 
\end{cases}
\end{align*}
Therefore, 
\begin{align*}
\mathrm{Nul}(tw)=
\begin{cases}
4, & t=t_1(a), t_2(a), \\
3, &  t\neq t_1(a), t_2(a). 
\end{cases}
\end{align*}
\end{Lemma}
\begin{proof}
Let $\alpha_1,\alpha_2$ be
\begin{align*}
&\alpha_1=\left\{\left(a-\frac{1}{4a}\right)+(a+\frac{1}{4a})e^{i\theta}\in \mathbb{\widehat{C}}_1 \mid 0 \leqslant \theta \leqslant\pi\right\}\notag\\
&\hspace{40mm}\cup \left\{\left(a-\frac{1}{4a}\right)+\left(a+\frac{1}{4a}\right)e^{i\theta}\in \mathbb{\widehat{C}}_2 \mid \pi \leqslant \theta \leqslant 2\pi\right\}, \\
&\alpha_2=\left\{\left(-\frac{1}{a}+\frac{a}{4}\right)+\left(\frac{1}{a}+\frac{a}{4}\right)e^{i\theta}\in \mathbb{\widehat{C}}_1 \mid 0 \leqslant \theta \leqslant 2\pi\right\}. 
\end{align*}
We take $\eta \in \widehat{H}(w)$ and express it as  
$$\eta=\gamma_1\eta_1+\gamma_2\eta_2+\gamma_3\eta_3, \,\,\gamma_j=x_j+iy_j,\,\,x_j,y_j\in\mathbb{R},\,\, j=1, 2 , 3. $$
We will find the simultaneous equation that 
$\gamma_1,\gamma_2,\gamma_3$
must satisfy so that the condition
$$\mathrm{Re}\int_{\alpha}(1-(tw)^2, i(1+(tw)^2), 2tw)\eta =0\,\, \mbox{for} \,\,\forall \alpha \,\,\mbox{closed curve in}\,\, M_a$$
holds for this $\eta.$ 

To do this, we first compute 
the integrals of $\eta_i, w\eta_i, w^2\eta_i,\,\,i=1,2,3, $ on $\alpha_1,\alpha_2$. When we actually compute, we change $\alpha_1,\alpha_2$ and we compute along the closed intervals $[0,a],[-\frac{1}{a},0]$ on the real axis, respectively. Since  the denominator of $\eta_i, \,\,i=1,2,3,$ has $z-A_1, z-A_2$ and these integrals diverge at $z=A_1,A_2, $ we subtract from $\eta_i$ the differentials of  meromorphic functions $f_i$ with poles at most order one at $(z,w)=(A_1,B_1),(A_2,B_2)$ on the Riemann surface $M_a$  so that the integrals of the meromorphic differentials $\eta_i-df_i$ converge. 

When we consider $\eta_1$, we obtain 
\begin{align*}
&\eta_1-\frac{1}{2B_1^2}d\frac{w}{(z-A_1)(z-A_2)}\\
=&\frac{z-\frac{1}{2}(A_1+A_2)}{(z-A_1)^2(z-A_2)^2}\frac{dz}{w}-\frac{1}{2B_1^2}\left(\frac{dw}{(z-A_1)(z-A_2)}-\frac{(z-A_1)w+(z-A_2)w}{(z-A_1)^2(z-A_2)^2}dz\right)\\
=&-\frac{3}{4B_1^2}\frac{dz}{w}+\frac{1}{B_1^2}\frac{(z-\frac{1}{2}(A_1+A_2))(w^2+B_1^2)}{(z-A_1)^2(z-A_2)^2}\frac{dz}{w}\\
=&-\frac{3}{4B_1^2}\frac{dz}{w}+\frac{1}{B_1^2}\frac{(z-\frac{1}{2}(A_1+A_2))(z+2A_1-(a-\frac{1}{a}))}{(z-A_2)^2w}dz\\
=&\frac{1}{4B_1^2}\frac{dz}{w}-\frac{(A_1-A_2)^2}{4B_1^2(z-A_2)^2w}dz.  
\end{align*} 
Therefore, by $\int_{\alpha_1}d\frac{w}{(z-A_1)(z-A_2)}=0, $
\noindent
\begin{eqnarray}\label{eqnarray-eta}
&&\int_{\alpha_1}\eta_1\notag\\
&=&\frac{1}{4B_1^2}\int_{\alpha_1}\frac{dz}{w}-\frac{(A_1-A_2)^2}{4B_1^2}\int_{\alpha_1}\frac{1}{(z-A_2)^2w}dz \notag\\
&=&\frac{i}{2B_1^2}\int_{0}^{a}\frac{dt}{\sqrt{t(a-t)(\frac{1}{a}+t)}}-\frac{i(A_1-A_2)^2}{2B_1^2}\int_{0}^{a}\frac{dt}{(t-A_2)^2\sqrt{t(a-t)(\frac{1}{a}+t)}}.  
\end{eqnarray}

We compute the rightmost side of \eqref{eqnarray-eta}. By the definition of the perfect elliptic integral, 
\begin{eqnarray}\label{eqnarray-E}
E(ia)=\int_{0}^{\frac{\pi}{2}}\sqrt{1+a^2\sin^2\theta}d\theta, 
\end{eqnarray}
\begin{eqnarray}\label{eqnarray-K}
K(ia)=\int_{0}^{\frac{\pi}{2}}\frac{d\theta}{\sqrt{1+a^2\sin^2\theta}}.  
\end{eqnarray}
Let $a\sin^2\theta=t$. Then $d\theta=\frac{dt}{2\sqrt{t(a-t)}}$, and \eqref{eqnarray-E}, \eqref{eqnarray-K} becomes 
\begin{eqnarray}\label{eqnarray-t}
E(ia)=\int_{0}^{a}\frac{\sqrt{1+at}}{2\sqrt{t(a-t)}}dt,  
\end{eqnarray}
\begin{eqnarray}\label{eqnarray-tt}
K(ia)=\int_{0}^{a}\frac{dt}{2\sqrt{t(a-t)(1+at)}}. 
\end{eqnarray}
 Also, 
\begin{eqnarray}\label{eqnarray-E-K}
E(ia)-K(ia)=\int_{0}^{a}\frac{at}{2\sqrt{t(a-t)(1+at)}}dt
\end{eqnarray}
holds. By \eqref{eqnarray-tt}, the first definite integral of the rightmost side of \eqref{eqnarray-eta} is 
\begin{align}\label{right1}
\int_{0}^{a}\frac{dt}{\sqrt{t(a-t)(\frac{1}{a}+t)}}=2\sqrt{a}K(ia). 
\end{align}

Next, we compute the second definite integral of the rightmost side of \eqref{eqnarray-eta}. If 
we set 
$$\varphi(t)=t^3+\left(\frac{1}{a}-a\right)t^2-t, \,\,
I[m]=\int_{0}^{a}\frac{t^mdt}{\sqrt{\varphi(t)}}, \,\,
J[m]=\int_{0}^{a}\frac{dt}{(t-A_2)^m\sqrt{\varphi(t)}}, $$
the recurrence formula 
\begin{align*}
&2m\varphi(A_2)J[m+1]+(2m-1)\varphi^\prime(A_2)J[m]+(m-1)\varphi^{\prime\prime}(A_2)J[m-1]\notag\\
&\hspace{30mm}+\frac{2m-3}{6}\varphi^{\prime\prime\prime}(A_2)J[m-2]+\frac{m-2}{12}\varphi^{\prime\prime\prime\prime}(A_2)J[m-3]=0
\end{align*}
holds (see \cite{J}). Since 
\begin{align*}
\varphi(A_2)=B_2^2, \quad \varphi^{\prime}(A_2)=0,\quad \varphi^{\prime\prime\prime}(A_2)=6,\quad\varphi^{\prime\prime\prime\prime}(A_2)=0, 
\end{align*}
this recurrence formula becomes 
\begin{align}\label{m}
&2mB_2^2J[m+1]+(2m-3)J[m-2]=0. 
\end{align}
Substituting $m=1$ into \eqref{m}, we obtain the following equation 
\begin{align}\label{o1}
2B_2^2J[2]=J[-1]. 
\end{align}
On the other hand, 
\begin{align}\label{J-1}
J[-1]=I[1]-A_2I[0]. 
\end{align}
By \eqref{eqnarray-t} and \eqref{eqnarray-E-K},  we obtain 
\begin{align}
&I[0]=\int_{0}^{a}\frac{dt}{\sqrt{t(a-t)(\frac{1}{a}+t)}}=2\sqrt{a}K(ia), \label{I0}\\
&I[1]=\int_{0}^{a}\frac{t}{\sqrt{t(a-t)(\frac{1}{a}+t)}}dt=\frac{2(E(ia)-K(ia))}{\sqrt{a}}\label{I1}.  
\end{align}
Substituting \eqref{I0} and \eqref{I1} into \eqref{J-1}, we obtain the following equation 
\begin{align}\label{o2}
J[-1]=2\frac{E(ia)-K(ia)-aA_2K(ia)}{\sqrt{a}}. 
\end{align}
Since the second definite integral of the rightmost side of \eqref{eqnarray-eta} is $J[2]$, 
\begin{align*}
\int_{0}^{a}\frac{dt}{(t-A_2)^2\sqrt{t(a-t)(\frac{1}{a}+t)}}=\frac{E(ia)-K(ia)-aA_2K(ia)}{\sqrt{a}B_2^2} 
\end{align*}
by \eqref{o1} and \eqref{o2}. 
Using this, the integral of $\eta_1$ on $\alpha_1$ is obtained as  
\begin{align*}
\int_{0}^{a}\eta_1=\frac{\sqrt{a}}{B_1^2}iK(ia)-\frac{(A_1-A_2)^2}{2B_1^2}\frac{\frac{1}{\sqrt{a}}(E(ia)-K(ia))-\sqrt{a}A_2K(ia)}{B_2^2}i. 
\end{align*}

Other integrals can be computed similarly (see Appendix). The results are as follows.  
\begin{align*}
&\int_{\alpha_1}\eta_1=\frac{\sqrt{a}}{B_1^2}iK(ia)-\frac{(A_1-A_2)^2}{2B_1^2}iI_2(a),\quad \int_{\alpha_1}w\eta_1=0,\quad \int_{\alpha_1}w^2\eta_1=3\sqrt{a}iK(ia), \\
&\int_{\alpha_1}\eta_2=\frac{2A_1\sqrt{a}}{B_1^2}iK(ia)+\frac{A_1+A_2}{3B_1^2}iI_2(a),\quad \int_{\alpha_1}w\eta_2=0, \quad\int_{\alpha_1}w^2\eta_2=\frac{6}{\sqrt{a}}iI_1(a), \\
&\int_{\alpha_1}\eta_3=0,\quad \int_{\alpha_1}w\eta_3=3\sqrt{a}iK(ia), \quad \int_{\alpha_1}w^2\eta_3=0, 
\end{align*}
\begin{align*}
&\int_{\alpha_2}\eta_1=-\frac{1}{\sqrt{a}B_2^2}K\left(\frac{i}{a}\right)+\frac{(A_1-A_2)^2}{2B_2^2}J_2(a),\quad \int_{\alpha_2}w\eta_1=0, \quad \int_{\alpha_2}w^2\eta_1=\frac{3}{\sqrt{a}}K\left(\frac{i}{a}\right), \\
&\int_{\alpha_2}\eta_2=-\frac{2A_2}{\sqrt{a}B_2^2}K\left(\frac{i}{a}\right)-\frac{A_1+A_2}{3B_2^2}J_2(a),\quad \int_{\alpha_2}w\eta_2=0, \quad \int_{\alpha_2}w^2\eta_2=-6\sqrt{a}J_1(a),  \\
&\int_{\alpha_2}\eta_3=0,\quad \int_{\alpha_2}w\eta_3=\frac{3}{\sqrt{a}}K\left(\frac{i}{a}\right), \quad \int_{\alpha_2}w^2\eta_3=0. 
\end{align*}
Here, 
\begin{align*}
&I_1(a)=E(ia)-K(ia),\quad I_2(a)=\frac{1}{B_2^2}\left(\frac{1}{\sqrt{a}}I_1(a)-\sqrt{a}A_2K(ia)\right),\\
&J_1(a)=E\left(\frac{i}{a}\right)-K\left(\frac{i}{a}\right),\quad J_2(a)=\frac{1}{B_1^2}\left(\sqrt{a}J_1(a)+\frac{1}{\sqrt{a}}A_1K\left(\frac{i}{a}\right)\right). 
\end{align*}

Now we return to the period condition 
$$
\mathrm{Re}\int_{\alpha}(1-(tw)^2, i(1+(tw)^2), 2tw)\eta =0 
$$
for $\eta = \gamma_1\eta_1+\gamma_2\eta_2+\gamma_3\eta_3\in \widehat{H}(w)$ 
and closed curves $\alpha=\alpha_1, \alpha_2$ on $M_a$.
First we obtain 
\begin{align*}
&\mathrm{Re}\int_{\alpha_1}w\eta =0\,\, \Longleftrightarrow\,\, \mathrm{Im}(\gamma_3)=0, \\
&\mathrm{Re}\int_{\alpha_2}w\eta =0\,\, \Longleftrightarrow\,\, \mathrm{Re}(\gamma_3)=0. 
\end{align*}
Therefore, $\gamma_3=0. $

Next, we have  
\begin{eqnarray}\label{eqnarray-a}
&&\mathrm{Re}\int_{\alpha}(1-w^2)\eta= \mathrm{Re}\int_{\alpha}i(1+w^2)\eta=0\nonumber\\
&\Longleftrightarrow&\int_{\alpha}\eta=\overline{\int_{\alpha}w^2\eta}. 
\end{eqnarray}
When we compute the left and right sides of \eqref{eqnarray-a} on $\alpha_1,\alpha_2$, we obtain 
\begin{eqnarray*}\label{eqnarray-b}
&&\int_{\alpha_1}\eta\\
&=&\gamma_1\left(\frac{\sqrt{a}}{B_1^2}iK(ia)-\frac{(A_1-A_2)^2}{2B_1^2}iI_2(a)\right)+\gamma_2\left(\frac{2A_1\sqrt{a}}{B_1^2}iK(ia)+\frac{A_1+A_2}{3B_1^2}iI_2(a)\right), \\
&&\overline{\int_{\alpha_1}w^2\eta}
=-3t^2\overline{\gamma_1}\sqrt{a}iK(ia)-\frac{6}{\sqrt{a}}t^2\overline{\gamma_2}iI_1(a), 
\end{eqnarray*} 
\begin{eqnarray*}
&&\int_{\alpha_2}\eta\\
&=&\gamma_1\left(\frac{-1}{\sqrt{a}B_2^2}K\left(\frac{i}{a}\right)+\frac{(A_1-A_2)^2}{2B_2^2}J_2(a)\right)+\gamma_2\left(\frac{-2A_2}{\sqrt{a}B_2^2}K\left(\frac{i}{a}\right)-\frac{A_1+A_2}{3B_2^2}J_2(a)\right), \\
&&\overline{\int_{\alpha_2}w^2\eta}=2t^2\overline{\gamma_1}\frac{3}{2\sqrt{a}}K\left(\frac{i}{a}\right)-6\sqrt{a}t^2\overline{\gamma_2}J_1(a).
\end{eqnarray*} 
Therefore, \eqref{eqnarray-a} with $\alpha=\alpha_1, \alpha_2$ becomes 
\begin{equation}\label{equation-t^2}
\left\{
\begin{aligned}
&\gamma_1(\frac{\sqrt{a}}{B_1^2}K(ia)-\frac{(A_1-A_2)^2}{2B_1^2}I_2(a))+\gamma_2(\frac{2A_1\sqrt{a}}{B_1^2}K(ia)+\frac{A_1+A_2}{3B_1^2}I_2(a))\\
& \hspace{70mm} =-3t^2\overline{\gamma_1}\sqrt{a}K(ia)-\frac{6t^2}{\sqrt{a}}\overline{\gamma_2}I_1(a), \\
&\gamma_1(-\frac{1}{\sqrt{a}B_2^2}K(\frac{i}{a})+\frac{2(A_1-A_2)^2}{4B_2^2}J_2(a))+\gamma_2(-\frac{2A_2}{\sqrt{a}B_2^2}K(\frac{i}{a})-\frac{A_1+A_2}{3B_2^2}J_2(a))\\
& \hspace{70mm} =t^2\overline{\gamma_1}\frac{3}{\sqrt{a}}K(\frac{i}{a})-6\sqrt{a}t^2\overline{\gamma_2}J_1(a). 
\end{aligned}
\right.
\end{equation} 
Let
\begin{eqnarray*}
  U = \left(
    \begin{array}{cc}
    \frac{1}{B_1^2}K(ia)-\frac{(A_1-A_2)^2}{2\sqrt{a}B_1^2}I_2(a) +3t^2K(ia)&\frac{2A_1}{B_1^2}K(ia)+\frac{A_1+A_2}{3B_1^2\sqrt{a}}I_2(a)+\frac{6t^2}{a}I_1(a)\\
    \frac{1}{B_2^2}K(\frac{i}{a})-\frac{\sqrt{a}(A_1-A_2)^2}{2B_2^2}J_2(a)+3t^2K\left(\frac{i}{a}\right)&\frac{2A_2K(\frac{i}{a})}{B_2^2}+\frac{A_1+A_2}{3B_2^2}\sqrt{a}J_2(a)-6at^2J_1(a) 
    \end{array}
  \right) 
  \end{eqnarray*}
  and
 \begin{align*}
    V= \left(
    \begin{array}{cc}
    \frac{1}{B_1^2}K(ia)-\frac{(A_1-A_2)^2}{2\sqrt{a}B_1^2}I_2(a) -3t^2K(ia)&\frac{2A_1K(ia)}{B_1^2}+\frac{A_1+A_2}{3\sqrt{a}B_1^2}I_2(a)-\frac{6t^2}{a}I_1(a)\\
    \frac{1}{B_2^2}K(\frac{i}{a})-\frac{\sqrt{a}(A_1-A_2)^2}{2B_2^2}J_2(a)-3t^2K\left(\frac{i}{a}\right)&\frac{2A_2K(\frac{i}{a})}{B_2^2}+\frac{A_1-A_2}{3\sqrt{a}B_2^2}J_2(a)+6at^2J_1(a)  
    \end{array}
  \right). 
 \end{align*}
If we use these matrices, then $\eqref{eqnarray-b}$ becomes  
\begin{equation}\label{equation-c}
\begin{pmatrix}
U & \text{\Large 0}\\
\text{\Large 0} & V
\end{pmatrix} \left(
    \begin{array}{c}
      x_1\\
     x_2 \\
     y_1\\
     y_2\\ 
    \end{array}
  \right)=0,  
\end{equation}
where $\gamma_j=x_j+iy_j,\,\,x_j,y_j\in\mathbb{R},\,\, j=1, 2$.  
$\eqref{equation-c}$ has a nontrivial solution if and only if $\det U=0$ or $\det V=0$. The conditional equation $\det U=0$ is a quadratic equation of $x=t^2$. This has only one positive real solution for each $a$. If we denote the positive square root of this positive real solution by $t_1(a)$, then this is the only positive real $t$ that satisfies $\det U=0$. The conditional equation $\det V=0$ is also a quadratic equation of $x=t^2$. This has only one positive real solution for each $a$. If we denote the positive square root of this positive real solution by $t_2(a)$, then this is the only positive real $t$ that satisfies $\det V=0$. If we draw the graphs of $t_1(a), t_2(a)$ in the range $1\leqslant a\leqslant a_0$ ( where $a_0$ can be numerically evaluated ) by Mathematica, we see that $t_1(a) < t_2(a)$.  When $t=t_1(a)$, $\det U=0$, but  $\det V\neq 0$. Then  there is only one nontrivial solution of $\eqref{equation-c}$ up to real multiple, and this nontrivial solution is written as $(x_1,x_2,y_1,y_2)=(a_1, a_2,0,0)$. When $t=t_2(a)$, $\det V=0$, but $\det U\neq 0$. Then  there is only one nontrivial solution of $\eqref{equation-c}$ up to real multiple, and this nontrivial solution is written as $(x_1,x_2,y_1,y_2)=(0,0,b_1,b_2)$.             
 Therefore, the basis of $H(t_1(a)w)$ is $\eta=a_1\eta_1+a_2\eta_2$ when $t=t_1(a)$ and the basis of $H(t_2(a)w)$ is $\eta=ib_1\eta_1+ib_2\eta_2$ when $t=t_2(a).$ Thus $\dim_\mathbb{R}H(tw)=1$ when $t=t_1(a),t_2(a)$, and $\dim_\mathbb{R}H(tw)=0$ when $t\neq t_1(a),t_2(a)$. Therefore, $\mathrm{Nul}(tw)=4$ when $t=t_1(a),t_2(a)$, and $\mathrm{Nul}(tw)=3$ when $t\neq t_1(a),t_2(a). $
\end{proof}

\begin{Remark}
By using Mathematica, we can check that $t_1(a)<t_2(a)$ for small values of $a$. In fact, as the graphs of Figure 1 suggest, the constant $a_0$ in the statement of Lemma \ref{dimH} is surely larger than 5. On the other hand, the graphs of Figure 2 suggest that the values of $t_1(a)$ and $t_2(a)$ become close to each other rather quickly as the parameter $a$ becomes bigger. 
\end{Remark}

\vspace{1.0cm}
\begin{figure}[h]
\includegraphics[width=10cm]{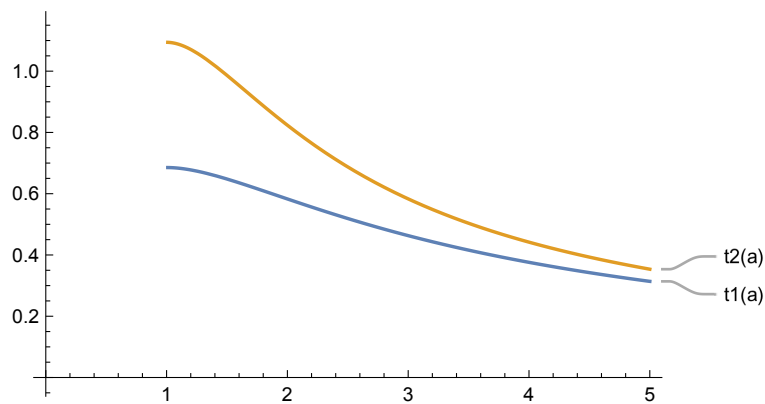}
\caption{Graphs of $t_1$, $t_2$ for $a\leqslant 5$.}
\label{}
\end{figure}
\vspace{1.0cm}

\begin{figure}[h]
\includegraphics[width=10cm]{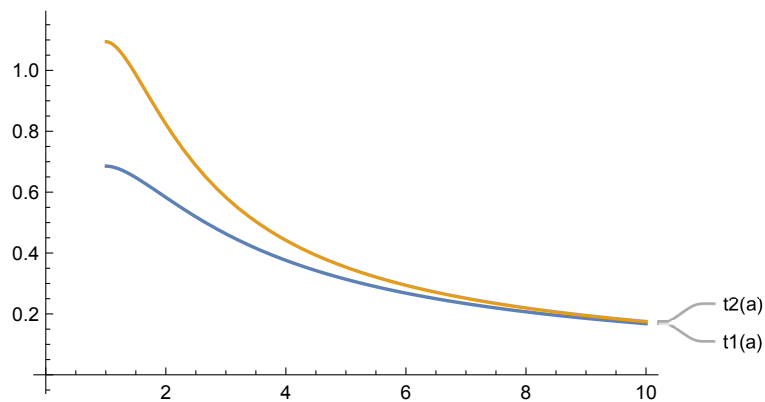}
\caption{Graphs of $t_1$, $t_2$ for $a\leqslant 10$.}
\label{}
\end{figure}

\noindent
\section{\bf Computation of $\mathrm{Ind}(tw)$}

In this section we compute the index of $tg_a$ for all $t$ and $a$ in the range
$1\leqslant a \leqslant a_0$, where $a_0$ is as in Section 4. 

As mentioned in Introduction, Nayatani $\cite{Nayatani}$ computed the index and nullity of 
$t \wp^\prime$, where $\wp$ is the Weierstrass $\wp$-function corresponding to the square lattice $\mathbb{Z}\oplus i\mathbb{Z}$. 
The Riemann surface $\mathbb{C}/\mathbb{Z}\oplus i\mathbb{Z}$ is isomorphic to 
\begin{align*}
M_1=\left\{(z,w)\in\widehat{\mathbb{C}}^2 \mid w^2=z(z^2-1)\right\}
\end{align*}
and $\wp^\prime$ coincides with $w\colon M_1\to \widehat{\mathbb{C}}$ up to a multiplicative positive real constant. 
Since we use Nayatani's result in the proof of Theorem \ref{theorem-d}, we state his result 
in our setting. 

\begin{Theorem}[Nayatani$\cite{Nayatani}$]\label{N-result}\rm
For the meromorphic function $w: M_1\ni (z,w)\longmapsto w\in\widehat{\mathbb{C}}$, 
we have 
\begin{align*}
\mathrm{Ind}(tw)=
\begin{cases}
5, & 0 < t\leqslant t_1(1),t_2(1)\leqslant t, \\
6, &  t_1(1)< t< t_2(1), 
\end{cases}
\end{align*}
and
\begin{align*}
\mathrm{Nul}(tw)=
\begin{cases}
4, &  t=t_1(1),t_2(1), \\
3, &  t\neq t_1(1),t_2(1).  
\end{cases}
\end{align*}
\end{Theorem}

\noindent
\begin{Theorem}\rm\label{theorem-d}
If $t_1(a)$ and $t_2(a)$ are as in Lemma \ref{dimH}, then
\begin{align*}
\mathrm{Ind}(tw)=
\begin{cases}
5, & 0 < t\leqslant t_1(a),t_2(a)\leqslant t, \\
6, &  t_1(a)< t< t_2(a) 
\end{cases}
\end{align*}
 for any $a$ in the range $1\leqslant a \leqslant a_0$ ( where $a_0$ can be numerically evaluated ). 
\end{Theorem}
\noindent
\begin{proof}
Let $g=tw$. We consider at $t=t_1(a)$. What we already know is $\mathrm{Nul}(g)=4, \mathrm{Ind}(g)=5$ when $t=t_1(1)$. That is,  there are exactly $5$ eigenvalues smaller than $2$ of $-\Delta_g$. If $a$ moves in the range $1\leqslant a \leqslant a_0$, then $\mathrm{Nul}(g)=4$ for all $a$ in this range. We arrange the eigenvalues of $-\Delta_g$  from the smallest, and we write the $i$-th eigenvalue as $\lambda_i(a),\,\, i=1,2,\dotsb$. When $a$ moves from $1$ to $a_0$, $\lambda_i(a)$ changes continuously, so if $\mathrm{Ind}(g)$ changes, then $\mathrm{Nul}(g)$ also changes. Therefore $\mathrm{Ind}(g)=5$ does not change. When $t=t_2(a)$, $\mathrm{Ind}(g)$ can also be determined similarly.  

Now, take one $a$ in the range from $1$ to $a_0$, and we write it as $a=\widetilde{a}_0$. We move $t$. When $t=t_1(\widetilde{a}_0), t_2(\widetilde{a}_0)$, then $\mathrm{Nul}(g)=4, \mathrm{Ind}(g)=5$. We consider $t=\frac{1}{2}t_1(\widetilde{a}_0)$ in the range from $0$ to $t_1(\widetilde{a}_0)$. When $t=\frac{1}{2}t_1(1)$, then $\mathrm{Nul}(g)=3, \mathrm{Ind}(g)=5$. If $a$ moves from $1$ to $a_0$, $\mathrm{Ind}(g)=5$ also does not change since $\mathrm{Nul}(g)=3$ does not change in the range from $0$ to $t_1(a)$. Therefore, when $t=\frac{1}{2}t_1(\widetilde{a}_0)$, then $\mathrm{Nul}(g)=3, \mathrm{Ind}(g)=5$. $\mathrm{Ind}(g)=5$ also does not change since $\mathrm{Nul}(g)=3$ does not change in the range from $0$ to $t_1(\widetilde{a}_0)$. Therefore, when $0< t< t_1(\widetilde{a}_0)$, then $\mathrm{Ind}(g)=5$. 
Similarly, we can show that $\mathrm{Ind}(g)=6$ when $t_1(\widetilde{a}_0)< t< t_2(\widetilde{a}_0)$, and $\mathrm{Ind}(g)=5$ when $t < t_2(\widetilde{a}_0)$. 
\end{proof}

We close this section with concluding remarks.
\begin{Remark} (i)\,\, Our argument for the proof of Theorem \ref{theorem-d} is diffrent from Nayatani's one in \cite{Nayatani}. 
Nayatani first showed $\mathrm{Nul}(tw)=3$ and $\mathrm{Ind}(tw)=5$ when t is sufficiently small, showed  $\mathrm{Nul}(tw)=4$ when $t=t_1(1),t_2(1)$ and $\mathrm{Nul}(tw)=3$ when $t\neq t_1(1),t_2(1)$. It follows that $\mathrm{Ind}(tw)=5$ since $\mathrm{Nul}(tw)=3$ does not change when $0<t<t_1(1)$. Next, he showed that one eigenvalue lager than $2$ become smaller than $2$ when $t$ pass through $t_1(1)$. Therefore, he could show $\mathrm{Ind}(tw)=5$ when $t=t_1(1)$ and$\mathrm{Ind}(tw)=6$ when $t=t_1(1)<t<t_2(1)$. Similar argument determines $\mathrm{Ind}(tw)$ for $t\geq t_2(1)$. We use this result to prove Theorem \ref{theorem-d}. 

\noindent
(ii)\,\, 
Since $H(tw)\neq \{0\}$ for $t=t_1(a), t_2(a)$, there exists a (possibly branched) complete orientable minimal surface in $\R^3$ whose extended Gauss map is $tw$ and all of whose ends are planer for each of $t=t_1(a), t_2(a)$. In particular, the Morse indices of these minimal surfaces are both $5$. 
If $t\neq t_1(a), t_2(a)$, $tw$ is still the extended minimal surface of {\em some} complete orientable minimal surfaces in $\R^3$, and Theorem \ref{theorem-d} computes the Morse indices of these minimal surfaces.

\noindent
(iii)\,\, 
In the case that the Riemann surface has genus zero, it is a remarkable result of Ejiri-Kotani \cite{Ejiri-Kotani} and Montiel-Ros \cite{Montiel-Ros} that the index of a generic meromorphic function of degree $d$ has index $2d-1$. On the other hand, in the higher-genus case, there are not so many complete orientable minimal surfaces nor meromorphic functions whose indices are computed. 
Theorem \ref{theorem-d} should be of some interest as it provides new examples of meromorphic functions on compact Riemann surfaces of genus $1$ whose indices are computable. 
\end{Remark}

\noindent
\section{Appendix}
We record computations omitted in Section $4$ here. 
We compute the integrals of $\eta_i, w\eta_i, w^2\eta_i,\,\,i=1,2,3, $ on $\alpha_1,\alpha_2$. 
As before,  
we compute along the closed intervals $[0,a],[-\frac{1}{a},0]$ on the real axis, respectively.  
Also, we subtract from $\eta_i$ differentials of  meromorphic functions $f_i$ 
on the Riemann surface $M_a$ so that the integrals of the meromorphic differentials 
$\eta_i-df_i$ converge.

\smallskip\noindent
(i)\,\, The integrals of $w\eta_1$ on $\alpha_1$.

 First we compute
\begin{flalign*}
\hspace{10mm}&w\eta_1+\frac{1}{2}d\frac{1}{(z-A_1)(z-A_2)}&&\\
=&\frac{z-\frac{1}{2}(A_1+A_2)}{(z-A_1)^2(z-A_2)^2}dz-\frac{z-\frac{1}{2}(A_1+A_2)}{(z-A_1)^2(z-A_2)^2}dz&&\\
=&0.  
\end{flalign*}
Therefore, by $\int_{\alpha_1}\frac{1}{2}d\frac{1}{(z-A_1)(z-A_2)}=0$,  
\begin{flalign*}
\hspace{10mm}\int_{\alpha_1}w\eta_1=0.  &&
\end{flalign*}

\smallskip\noindent
(ii)\,\, The integrals of $w^2\eta_1$ on $\alpha_1$. 

First we compute
\begin{flalign*}
\hspace{10mm}&w^2\eta_1+\frac{1}{2}d\frac{w}{(z-A_1)(z-A_2)}&&\\
=&\frac{(z-\frac{1}{2}(A_1+A_2))w}{(z-A_1)^2(z-A_2)^2}dz+\frac{dw}{2(z-A_1)(z-A_2)}-\frac{(z-A_1)w+(z-A_2)w}{2(z-A_1)^2(z-A_2)^2}dz&&\\
=&\frac{3}{4}\frac{dz}{w}.  
\end{flalign*}
Therefore, by $\int_{\alpha_1}d\frac{w}{(z-A_1)(z-A_2)}=0$,  
\begin{flalign*}
\hspace{10mm}\int_{\alpha_1}w^2\eta_1=\frac{3}{4}\int_{\alpha_1}\frac{dz}{w}=3\sqrt{a}iK(ia). &&
\end{flalign*}

\smallskip\noindent
(iii)\,\, The integrals of $\eta_2$ on $\alpha_1$. 

First we compute
\begin{flalign*}
\hspace{10mm}&\eta_2-\frac{A_1}{B_1^2}d\frac{w}{(z-A_1)(z-A_2)}&&\\
=&\frac{z^2-A_1A_2}{(z-A_1)^2(z-A_2)^2w}dz-\frac{A_1}{B_1^2}d\frac{w}{(z-A_1)(z-A_2)}&&\\
=&\frac{z^2-A_1A_2}{(z-A_1)^2(z-A_2)^2w}dz-\frac{A_1}{B_1^2}(\frac{dw}{(z-A_1)(z-A_2)}-\frac{(z-A_1)w+(z-A_2)w}{(z-A_1)^2(z-A_2)^2}dz)&&\\
=&-\frac{3A_1}{2B_1^2}\frac{dz}{w}+\frac{B_1^2(z^2-A_1A_2)+A_1(2z-(A_1+A_2))w^2}{B_1^2(z-A_1)^2(z-A_2)^2w}dz&&\\
=&-\frac{3A_1}{2B_1^2}\frac{dz}{w}+\frac{B_1^2(z^2-A_1A_2)+(z^2-A_1A_2)w^2-(z^2-2zA_1+A_1^2)w^2}{B_1^2(z-A_1)^2(z-A_2)^2w}dz&&\\
=&-\frac{3A_1}{2B_1^2}\frac{dz}{w}+\frac{(z^2-A_1A_2)(w^2+B_1^2)}{B_1^2(z-A_1)^2(z-A_2)^2w}dz-\frac{w}{B_1^2(z-A_2)^2}dz&&\\
=&-\frac{3A_1}{2B_1^2}\frac{dz}{w}+\frac{(z^2-A_1A_2)(z+2A_1-(a-\frac{1}{a}))}{B_1^2(z-A_2)^2w}dz-\frac{w}{B_1^2(z-A_2)^2}dz&&\\
=&-\frac{3A_1}{2B_1^2}\frac{dz}{w}+\frac{2A_1z^2+\frac{2}{3}(2z+A_1)-\frac{1}{3}(a-\frac{1}{a})}{B_1^2(z-A_2)^2w}dz&&\\
=&\frac{A_1}{2B_1^2}\frac{dz}{w}+\frac{A_1+A_2}{6B_1^2}\frac{dz}{(z-A_2)^2w}.  
\end{flalign*}
Therefore, by $\int_{\alpha_1}d\frac{w}{(z-A_1)(z-A_2)}=0$,  
\begin{flalign*}
\hspace{10mm}\int_{\alpha_1}\eta_2=&\frac{A_1}{2B_1^2}\int_{\alpha_1}\frac{dz}{w}+\frac{A_1+A_2}{6B_1^2}\int_{\alpha_1}\frac{dz}{(z-A_2)^2w}\\=&\frac{2A_1\sqrt{a}}{B_1^2}iK(ia)+\frac{A_1+A_2}{3B_1^2}iI_2(a). &&
\end{flalign*}

\smallskip\noindent
(iv)\,\, The integrals of $w\eta_2$ on $\alpha_1$. 

First we compute
\begin{flalign*}
\hspace{10mm}&w\eta_2+d\frac{z}{(z-A_1)(z-A_2)}&&\\
=&\frac{z^2-A_1A_2}{(z-A_1)^2(z-A_2)^2}dz+d\frac{z}{(z-A_1)(z-A_2)}&&\\
=&\frac{z^2-A_1A_2}{(z-A_1)^2(z-A_2)^2}dz+\frac{1}{(z-A_1)(z-A_2)}dz-\frac{z(z-A_1)+z(z-A_2)}{(z-A_1)^2(z-A_2)^2}dz&&\\
=&0. 
\end{flalign*}
Therefore, by $\int_{\alpha_1}d\frac{z}{(z-A_1)(z-A_2)}=0$, 
\begin{flalign*}
\hspace{10mm}\int_{\alpha_1}w\eta_2=0. &&
\end{flalign*}

\smallskip\noindent
(v)\,\, The integrals of $w^2\eta_2$ on $\alpha_1$. 

First we compute
\begin{flalign*}
\hspace{10mm}&w^2\eta_2+d\frac{zw}{(z-A_1)(z-A_2)}&&\\
=&\frac{w(z^2-A_1A_2)}{(z-A_1)^2(z-A_2)^2}dz+d\frac{zw}{(z-A_1)(z-A_2)}&&\\
=&\frac{w(z^2-A_1A_2)}{(z-A_1)^2(z-A_2)^2}dz+\frac{w}{(z-A_1)(z-A_2)}dz\\
&+\frac{zdw}{(z-A_1)(z-A_2)}-\frac{(z-A_1)w+(z-A_2)w}{(z-A_1)^2(z-A_2)^2}dz&&\\
=&\frac{w(z^2-A_1A_2)}{(z-A_1)^2(z-A_2)^2}dz+\frac{3z}{2w}dz\\
&+\frac{w(z^2-(A_1+A_2)z+A_1A_2-2z^2+(A_1+A_2)z)}{(z-A_1)^2(z-A_2)^2}dz&&\\
=&\frac{3z}{2w}dz. 
\end{flalign*}
Therefore, by $\int_{\alpha_1}d\frac{zw}{(z-A_1)(z-A_2)}=0$, 
\begin{flalign*}
\hspace{10mm}\int_{\alpha_1}w^2\eta_2=\int_{\alpha_1}\frac{3z}{2w}dz=\frac{6}{\sqrt{a}}iI_1(a). &&
\end{flalign*}

\smallskip\noindent
(vi)\,\, The integrals of $\eta_3$ on $\alpha_1$. 

First we compute
\begin{flalign*}
\hspace{10mm}&\eta_3+\frac{1}{2}d\frac{1}{(z-A_1)(z-A_2)}&&\\
=&\frac{z-\frac{1}{2}(A_1+A_2)}{(z-A_1)^2(z-A_2)^2}dz-\frac{z-\frac{1}{2}(A_1+A_2)}{(z-A_1)^2(z-A_2)^2}dz&&\\
=&0.  
\end{flalign*}
Therefore, by $\int_{\alpha_1}d\frac{1}{(z-A_1)(z-A_2)}=0$, 
\begin{flalign*}
\hspace{10mm}\int_{\alpha_1}\eta_3=0. &&
\end{flalign*}

\smallskip\noindent
(vii)\,\, The integrals of $w\eta_3$ on $\alpha_1$. 

First we compute 
\begin{flalign*}
\hspace{10mm}&w\eta_3+\frac{1}{2}d\frac{w}{(z-A_1)(z-A_2)}&&\\
=&\frac{(z-\frac{1}{2}(A_1+A_2))w}{(z-A_1)^2(z-A_2)^2}dz+\frac{dw}{2(z-A_1)(z-A_2)}-\frac{(z-A_1)w+(z-A_2)w}{2(z-A_1)^2(z-A_2)^2}dz&&\\
=&\frac{3}{4}\frac{dz}{w}.  
\end{flalign*}
Therefore, by $\int_{\alpha_1}d\frac{w}{(z-A_1)(z-A_2)}=0$, 
\begin{flalign*}
\hspace{10mm}\int_{\alpha_1}w\eta_3=\frac{3}{4}\int_{\alpha_1}\frac{dz}{w}=\frac{6\sqrt{a}}{2}iK(ia). &&
\end{flalign*}

\smallskip\noindent
(viii)\,\, The integrals of $w^2\eta_3$ on $\alpha_1$. 

First we compute 
\begin{flalign*}
\hspace{10mm}&w^2\eta_3+\frac{B_1^2}{2}d\frac{1}{(z-A_1)(z-A_2)}&&\\
=&\frac{(z-\frac{1}{2}(A_1+A_2))(w^2-B_1^2)}{(z-A_1)^2(z-A_2)^2}dz&&\\
=&\frac{(z-\frac{1}{2}(A_1+A_2))(z+2A_1-(a-\frac{1}{a}))}{(z-A_1)^2(z-A_2)^2}dz\\
=&dz-\frac{(A_1-A_2)^2}{4}\frac{dz}{(z-A_2)^2}.  
\end{flalign*}
Therefore, by $\int_{\alpha_1}d\frac{1}{(z-A_1)(z-A_2)}=0$, 
\begin{flalign*}
\hspace{10mm}\int_{\alpha_1}w^2\eta_3=\int_{\alpha_1}dz-\frac{(A_1-A_2)^2}{4}\int_{\alpha_1}\frac{dz}{(z-A_2)^2}=0. &&
\end{flalign*}

\section*{}

\end{document}